\documentclass[a4paper]{article}
\usepackage{amsfonts}
\usepackage{amsmath}
\usepackage{amsbsy}
\usepackage{amsxtra}
\usepackage{latexsym}
\usepackage{amssymb}
\usepackage{verbatim}
\usepackage{graphicx}

\DeclareMathOperator{\cone}{cone}

\DeclareMathOperator{\spa}{span}

\newcommand{\R}{\mathbb R}

\newcommand{\E}{\mathbf E}

\newcommand{\U}{\mathbf U}

\newcommand{\x}{\mathbf x}
\newcommand{\y}{\mathbf y}
\newcommand{\z}{\mathbf z}

\newcommand{\ee}{\mathbf e}
\newcommand{\uu}{\mathbf u}
\newcommand{\vv}{\mathbf v}

\newcommand{\0}{\circ}
\newcommand{\Pp}{\mathbf P}

\newtheorem{remark}{Remark}
\newtheorem{theorem}{Theorem}
\newtheorem{lemma}{Lemma}
\newtheorem{corollary}{Corollary}

\newenvironment{proof}{{\bf Proof.}}{\hfill$\Box$\\}

\begin{document}

\title{Rapid heuristic projection on simplicial cones
\thanks{{\it 1991 A M S Subject Classification.} Primary 90C33;
Secondary 15A48, {\it Key words and phrases.} Metric projection  on  simplicial cones}}
%\author{\large A. Ek\'art, A. B. N\'emeth, S. Z. N\'emeth}
\author{A. Ek\'art\\Computer Science, Aston University\\Aston Triangle, Birmingham B4 7ET\\United Kingdom\\email: a.ekart@aston.ac.uk \and A. B. N\'emeth\\Faculty of Mathematics and Computer Science\\Babe\c s Bolyai University, Str. Kog\u alniceanu nr. 1-3\\RO-400084 Cluj-Napoca, Romania\\email: nemab@math.ubbcluj.ro \and S. Z. N\'emeth\\School of Mathematics, The University of Birmingham\\The Watson Building, Edgbaston\\Birmingham B15 2TT, United Kingdom\\email: nemeths@for.mat.bham.ac.uk}
\maketitle

\begin{abstract}
	A very fast heuristic iterative method of projection on simplicial cones is presented. It consists in 
	solving two linear systems at each step of the iteration. The extensive experiments indicate that the 
	method furnishes the 
	exact solution in more then 99.7 percent of the cases. The average number of steps
	is 5.67 (we have not found any examples which required more than 13 steps)
	and the relative number of steps with respect to the dimension decreases 
	dramatically. Roughly speaking, for high enough dimensions the absolute number of 
	steps is independent of the dimension.
\end{abstract}

\section{Introduction}

Projection  on polyhedral cones is one of the important problems applied optimization 
is confronted with. Many applied numerical optimization methods uses projection on 
polyhedral cones as the main tool.
 
In most of them, projection is part of an iterative 
process which involve its repeated application (see e. g. problems of image 
reconstruction \cite{Dattorro2005}, nonlinear complementarity 
\cite{IsacNemeth1990c,Nemeth2008}, etc.). Hence, it is important to get fast projection 
algorithms.

The main streems of the current methods in use rely on the classical von Neumann algorithm 
(see e.g. the Dykstra algorithm \cite{Dykstra1983,DeutschHundal1994,Shusheng2000}), but 
they are rather expensive for a numerical handling (see the numerical results in 
\cite{Morillas2005} and the remark preceding section 6.3 in 
\cite{MingGuo-LiangHong-BinKaiWang2007}).

Finite methods of projections are of combinatorial nature which reduces their 
applicability to low dimensional ambient spaces. 

Recently we have given a simple projection method exposed in note 
\cite{NemethNemeth2009} for projecting on so called isotone projection cones. Isotone 
projection cones are special simplicial cones, and due to their good properties we can 
project on them in $n$ steps, where $n$ is the dimension of the ambient space. In the 
first part of that note we have explained our approach by considering the problem of 
projection on simplicial cones by giving an exact method based on duality. This method has
combinatorial character and therefore it is inefficient. More recently we observed that a 
heuristic method based on the same ideas gives surprisingly good results. This note 
describes the theoretical foundation of the heuristic method and draws conclusions based 
on millions of numerical experiments. 

Projection on polyhedral cones is a problem of high impact on scientific community.\footnote{see the popularity of the 
Wikimization page \textit{Projection on Polyhedral Cone} at 
\vspace{2mm}

\noindent
{http://www.convexoptimization.com/wikimization/index.php/Special:Popularpages}.}

\section{The  simplicial cone and its polar}

Let $\R^n$ be an $n$-dimensional Euclidean space endowed with a Cartesian reference 
system. We assume that each point of $\R^n$ is a column vector. 

We shall use the term cone in the sense of closed convex cone. That is, the nonempty 
closed subset $K\subset \R^n$ in our terminology is a {\it cone}, if $K+K\subset K$, and  
$tK\subset K$ whenever $t\in \R,\; t\geq 0$. 

Let $m\leq n$ and $\ee_1,\dots,\ee_m$ be $m$ elements in $\R^n$. Denote
\begin{equation*}%\label{1}
	\cone \{\ee_1,\dots,\ee_m\}
	=\{\lambda^1\ee_1+\dots+\lambda^m\ee_m:\lambda^i\geq 0, \,i=1,\dots,m\},
\end{equation*}
the {\it cone engendered by} $\ee_1,\dots,\ee_m$.
Then, 
\begin{equation}\label{1b}
	\cone\{\ee_1,\dots,\ee_m\}=\{\E\vv:\vv\in\R^m_+\},
\end{equation}
where $\E=(\ee_1,\dots,\ee_m)$ is the matrix with columns $\ee_1,\dots,\ee_m$
and $\R^m_+$ is the non-negative orthant in $\R^m$.

Suppose that $\ee_1,\dots,\ee_n\in \R^n$ are linearly independent elements. Then, the cone
\begin{equation}\label{1}
	\begin{array}{rcl}
		K & = & \cone\{\ee_1,\dots,\ee_n\}\\ & = &
		\{\lambda^1\ee_1+\dots+\lambda^n\ee_n:\lambda^i\geq 0,\,i=1,\dots,n\}=
		\{\E\vv:\vv\in\R^m_+\},
	\end{array}
\end{equation}
with $\E$ the matrix from (\ref{1b}) for $m=n$, is called {\it simplicial cone}. Denote 
$N=\{1,2,\dots,n\}$.

The {\it polar} of $K$ is the set
\begin{equation}\label{2}
	K^\0 = \{\x\in \R^n:\,\x^\top\y\leq 0,\,\forall\y\in K\}.
\end{equation}
$K^*=-K^\0$ is called the {\it dual} of $K$. $K$ is called {\it subdual}, if 
$K\subset K^*$. This is equivalent to the condition 
$\ee_\ell^\top\ee_k\geq 0,\,\ell,k\in N.$  

\begin{lemma}\label{polar}
The polar of the  simplicial cone (\ref{1}) can be represented in the form
\begin{equation}\label{3}
	K^\0 =\{\mu^1\uu_1+\dots\mu^n\uu_n:\mu^i\geq 0,\,i=1,\dots,n\},
\end{equation}
where $\uu_i(i=1,\dots,n)$ is a solution of the system
$$\ee^\top_j\uu_i =0,\;j=1,\dots,n,\;j\neq i,$$
$$\ee^\top_i\uu_i=-1$$
($\uu_i$ is normal to the hyperplane
$\spa \{\ee_1,\dots,\ee_{i-1},\ee_{i+1},\dots,\ee_n\}$ in the opposite direction
to the halfspace that contains $\ee_i$). Thus, $$K^\0=\{\U\x:\x\in\R^n_+\}$$ with 
$\R^n_+=\{\x=(x_1,\dots,x_n)^\top: x_i\geq 0,\;i=1\dots n\}$ and 
\begin{equation}\label{polmat}
	\U=-(\E^{-1})^\top.
\end{equation} 
For simplicity we shall call $\U$ the \emph{polar matrix} of $\E$. The columns of $\U$ are 
$\{\uu_i$: $i=1,\dots,n\}$.
\end{lemma}
\begin{proof} 
Let $\y=\sum_{j=1}^n\mu^j \uu_j$ and $\z=\sum_{i=1}^n \alpha^i \ee_i$
for any non-negative real numbers $\alpha^i$ and $\mu ^j$. The inner
produce of $\y$ and $\z$ is non-positive because
$$\y^{\top}\z=\sum_{i=1}^n \sum_{j=1}^n \alpha ^i \mu ^j\ee_i^\top \uu_j=-\sum_{i=1}^n \alpha ^i \mu ^i\leq 0$$
But $\y$ is an arbitrary element of the right hand side of (\ref{3})
and $\z$ is an arbitrary element of $K$, thus we can conclude that the righ hand side of (\ref{3}) is 
a subset of $K^\0$.

The vectors $\uu_1,\dots,\uu_n$ are linearly independent.
(This can be verified by assuming the contrary, and by
multiplying the subsequent relation by $\ee_j$
to get a contradiction.)
Hence for $\y\in K^\0$ we have the representation
$$y=\beta^1\uu_1+\dots+\beta^n\uu_n.$$

By (\ref{2}), $$\ee^\top_k \y=-\beta^k\leq 0$$ so $\beta^k\geq 0$ which 
prove that $\y$ is an element of the right hand side of (\ref{3}). 
Thus we can conclude that $K^\0$ is a subset of the right hand side of 
(\ref{3}).
\end{proof}

The formula (\ref{polmat}) of Lemma \ref{polar} is equivalent to the formula (380) of 
\cite{Dattorro2005}. 

\begin{corollary}\label{eiujbaz}
For each subset $I$ of indices in $N$, the vectors $\ee_i,i\in I,\uu_j,j\in I^c$ 
(where $I^c$ the complement of $I$ with respect to $N$) are linearly independent.
\end{corollary}
\begin{proof}
Assume that
\begin{equation}\label{4} 
	\sum_{i\in I}\alpha^i\ee_i+\sum_{j\in I^c}\beta^j \uu_j=0
\end{equation}
for some reals $\alpha^i$ and $\beta^j.$ By the mutual orthogonality of the vectors
$\ee_i,i\in I$ and $\uu_j,j\in I^c$ it follows, by multiplication of
the relation (\ref{4}) with $\sum_{i\in I} \alpha^i \ee_i^\top$ and
respectively with $\sum_{j\in I^c}\beta^j \uu_j^\top$, that
$$\sum_{i\in I} \alpha^i \ee_i=0$$
and
$$\sum_{j\in I^c}\beta^j \uu_j=0.$$
Hence, $\alpha^i=\beta^j=0$ must hold.
\end{proof}

The cone $K_0\subset K$ is called a {\it face} of $K$ if from $\x\in K_0,\y\in K$ and $\x-\y\in K$
it follows that $\y\in K_0.$ The face $K_0$ is called a {\it proper face} of $K$, if $K_0\not=K.$

\begin{lemma}\label{Youdinfaces}
If $K$ is the cone (\ref{1}) and $K^\circ$ is the cone (\ref{3}), then for every subset
of indices $ I =\{i_1,...,i_k\}\subset  N$ the set
\begin{equation}\label{youdinfaces}
F_ I =\cone \{\ee_i:\;i\in  I\}= \{\x\in K:\;\x^\top \uu_j=0:\;i\in  I ^c\}
\end{equation}
(with $F_I=\{\mathbf 0\}$ if $I=\emptyset$) is a face of $K$. 
If $i_h\not= i_l$ whenever $h\not= l$, then $F_ I$ 
is for $k>0$ a nonempty set in $\R^n$  of dimension $k$.
(In the sense that $F_I$ spans a subspace of $\R^n$ of dimension $k$.) 

Every face of $K$ is equal to $F_ I$ for some $I\subset  N.$
If $I\not=N$ then $F_I$ is a proper face.
\end{lemma}

\begin{proof}

The relation in (\ref{youdinfaces}) follows from the definition of
the vectors $\uu_j$ in  Lemma \ref{polar}, while the
 assertion on the dimension of $F_ I$ is obvious.

Suppose that $\x\in F_ I$ and $\y\in K$ with $\y\leq \x$.

Then $(\x-\y)^\top \uu_j =-\y^\top \uu_j\leq 0$, 
$\forall j\in  I^c$ and $\y^\top \uu_j \leq 0,\;\forall j\in  N$, because $\y\in K$.
Thus $\y^\top \uu_j =0,\;\forall j\in  I^c$, hence $y\in F_ I,$
showing that $F_ I$ is a face.

Suppose that $\x\in F$ for $F$ an arbitrary proper face of $K$. Since $\x\in K$, by
the definition of the vectors $\uu_j,$  $\x^\top \uu_j \leq 0$ for $j\in  N.$

 If 
$\x^\top \uu_j <0,\;\forall j\in  N$, then there exists a positive
scalar $t$ with $(\x-t\y)^\top \uu_j \leq 0,\;\forall j\in  N$.
Hence, $\x-t\y\in K$ and thus $t\y\leq \x$. But then $t\y\in F$ and since $F$ is
a cone, $\y\in F$. This means that $K\subset F$, that is, $F$ cannot be a
proper face.

We have to show that $F$ has a representation like (\ref{youdinfaces}). By the above
reasoning, for each $\x\in F$ there exist some index  $i\in N$ with with
$\x^\top \uu_i=0.$

If $F=\{\mathbf 0\}$ we have the representation (\ref{youdinfaces}) with $I=\emptyset.$

If $F\not= \{\mathbf 0\}$, take $\x$ in the relative interior of $F$ and let $I$ be
the complement in $N$ of the maximal set of indices $j$ with $\x^\top \uu_j=0.$
($I$ must be a nonempty, proper subset of $N$ since $\x\not=\mathbf 0.$)

Take $\y\in F$ arbitrarily. By the definition of $\x$, $\x-t\y\in F$ for some sufficiently
small $t>0$. Hence,
\begin{equation}\label{kieg}
 (\x-t\y)^\top \uu_i \leq 0,\;\;\forall i\in N.
\end{equation}
By $\y\in F\subset K$ we also have $\y^\top \uu_i\leq 0,\;\forall i\in N.$
If $\y^\top \uu_j <0$ for some $j\in I^c$, then (\ref{kieg}) would imply
$$\x^\top \uu_j\leq t \y^\top \uu_j <0,$$
which is a contradiction. Hence, we must have $\y^\top \uu_j =0,\;\forall j\in I^c$; and accordingly
\begin{equation}\label{kiegg}
	F\subset \{\z\in K:\;\z^\top \uu_j = 0,\;\forall j\in I^c\}.
\end{equation}

Suppose that $\y \in K$ and $\y^\top \uu_j =0,\;\forall j\in I^c$. From definition 
we have  $\x^\top \uu_i <0$ for each $i\in I$, whereby for a sufficiently large $t>0$,
$$(t\x-\y)^\top \uu_i \leq 0,\;\forall i\in N.$$
Hence, $t\x-\y$ is in the polar of $K^\0$, which by Farkas' lemma is $K$ 
(This follows in fact, in our case, also by the symmetry of the vectors $\ee_i$ and
$\uu_j$ in the formulae of Lemma \ref{polar}.)
Thus
$\mathbf 0\leq \y\leq t \x$, whereby $\mathbf 0\leq (1/t) \y\leq \x.$  Since $F$ is a face of $K$,
we have $(1/t)\y \in F$ and since it is also a cone, $\y\in F.$
This proves the converse of the inclusion in (\ref{kiegg}) and completes the proof.

\end{proof}

Thus a maximal proper face of $K$ is of the form
$$K_{i_0}=\cone \{\ee_i:i\in N\setminus \{i_0\}\}=\cone \{\ee_i:i\in N,\ee^\top_i\uu_{i_0}=0\},$$
hence it is also called {\it the face of $K$ orthogonal to} $\uu_{i_0}$.
Similarly, we have a maximal proper face of $K^\0$ orthogonal to some $\ee_{j_0}$.

An equivalent result to the one presented in Lemma \ref{Youdinfaces} is given by the Cone Table 1 on 
page 179 of \cite{Dattorro2005}.

Thus a maximal proper face of $K$ is of the form
$$K_{i_0}=\cone \{\ee_i:i\in  N\setminus \{i_0\}\}=\cone \{\ee_i:i\in  N,\ee^\top_i\uu_{i_0}=0\},$$
hence it is also called {\it the face of $K$ orthogonal to} $\uu_{i_0}$.
Similarly, we have a maximal proper face of $K^\circ$ orthogonal to some $\ee_{j_0}$.

Let $F=\cone \{\ee_i:\;i\in  I\} $ and $F^\perp =\cone \{\uu_j:\;j\in  I^c\}.$ Then, from 
the above results it follows that 
$$F=\{\x\in K:\;\x^\top\uu_j=0,\,j\in  I^c\}$$
and
$$F^\perp=\{\y\in K^\circ:\;\y^\top\ee_i=0,\,i\in I\}.$$

The faces $F\subset K$ and $F^\perp \subset K^\circ$ of the above form
are called a {\it pair of orthogonal faces} 
where $F^\perp$ is called the {\it orthogonal face of $F$} and $F$ is called the 
{\it orthogonal face of $F^\perp$}.

\section{Finite method of projection  on a  simplicial cone}

For an arbitrary $\uu\in\R^n$ denote $\|\uu\|=\sqrt{\uu^\top\uu}$.
Let $K\in\R^n$ be an arbitrary cone and $K^\0$ its polar, and $C\subset\R^n$ an arbitrary closed convex 
set. Recall that the 
\emph{projection mapping} $\Pp_C:H\to H$  on $C$ is well defined by $\Pp_C\x\in C$ and 
\[\|\x-\Pp_C\x\|=\min\{\|\x-\y\|:\y\in C\}.\] 

Then, Moreau's decomposition theorem asserts:

\begin{theorem}\label{mor} (Moreau, \cite{Moreau1962})
For $\x,\,\y,\,\z\in \R^n$ the
following statements are equivalent:
\begin{enumerate}
	\item[(i)] $\z=\x+\y, \x\in K,\y\in K^\0$ and $\x^\top\y =0.$
	\item[(ii)] $\x=\Pp_K\z$ and $\y=\Pp_{K^\0}\z$.
\end{enumerate}
\end{theorem}

Suppose now, that $K$ is a  simplicial cone
in $\R^n$. We shall use the representation (\ref{1})
for $K$ and the representation (\ref{3}) for $K^\0$.
Hence, $$ \ee^\top_i\uu_j = -\delta^i_j, i,j=1,\dots,n$$
where $\delta^i_j$ the Kronecker symbol.
As a direct implication of Moreau's decomposition theorem
and the constructions in the preceding section we have:

\begin{theorem}\label{lattproj}
Let $\x\in \R^n$. For each subset of indices $I\subset N$,
$\x$ can be represented in the form 
\begin{equation}\label{5}
	\x=\sum_{i\in I}\alpha^i\ee_i+\sum_{j\in I^c}\beta^j\uu_j
\end{equation}
with $I^c$ the complement of $I$ with
respect to $N$, and with $\alpha^i$ and
$\beta^j$ real numbers.
Among the subsets $I$ of indices, there exists exactly 
one (the cases $I=\emptyset$ and $I=N$ are not excluded) 
with the property that for the coefficients in (\ref{5})
one has $\beta^j> 0, j\in I^c$ and $\alpha^i\geq 0,
i\in I.$ For this representation it holds that
\begin{equation}\label{6}
	\Pp_K\x=\sum_{i\in I} \alpha^i \ee_i,\quad \alpha^i\geq 0,
\end{equation}
and
\begin{equation}\label{7}
	\Pp_{K^\0}\x=\sum_{j\in I^c}\beta^j\uu_j,\quad \beta^j>0.
\end{equation}
\end{theorem}

\begin{proof}
The first assertion is the consequence of Corollary
\ref{eiujbaz}.

The projections $\Pp_K\x$ and $\Pp_{K^\0}\x$ as elements of
$K$ and $K^\0$, respectively can be represented as
\begin{equation}\label{8}
	\Pp_K\x=\sum_{i=1}^n \alpha^i\ee_i,\quad \alpha^i\geq 0
\end{equation}
and
\begin{equation}\label{9}
	\Pp_{K^\0}\x=\sum_{j=1}^n\beta^j\uu_j,\quad \beta^j\geq 0.
\end{equation}
To prove existence, let $I^c=\{j\in N:\beta^j>0\}$ and let $I$ be the complement of $I^c$ in the set $N$ of 
indices. For an arbitrary element $\z\in\R^n$, denote $\Pp_K^\top \z=(\Pp_K\z)^\top$. If 
$\alpha^j>0$ would hold in (\ref{9}), for some 
$j\in I^c$, then by Lemma \ref{polar} it would follow that 
$\Pp_K^\top\x\cdot\Pp_{K^\0}\x<0,$ which contradicts the theorem of Moreau. Hence, 
(\ref{8}) can be written in the form (\ref{6}) and (\ref{9}) can be written in the form 
(\ref{7}). Therefore, Theorem \ref{mor} implies 
\[\x=\Pp_K\x+\Pp_{K^\0}\x=\sum_{i\in I} \alpha^i \ee_i+\sum_{j\in I^c}\beta^j\uu_j,\]
where $\alpha^i\geq 0$, $\forall i\in I$ and $\beta^j>0$, $\forall j\in I^c$. 

To prove uniqueness, suppose that in the representation (\ref{5}) of $\x$ we have 
$\alpha ^i\geq 0 ,\;\beta^i=0$ for $i\in I$ and $\beta ^j>0,\;\alpha ^j=0$ for $j\in I^c$, where
$I$ is a subset of $N$, and $I^c$ is the complement 
of $I$ in $N$ (the cases $I=\emptyset$ and $I=N$ are not excluded). Then representations (\ref{6}) and
(\ref{7}) follow from Theorem \ref{mor} by using the
mutual orthogonality of the vectors $\ee_i,\;i\in I$ and $\uu_j,\;j\in I^c$.
From (\ref{6}) and the uniqueness of the projection $\Pp_K\x$ it follows that
$I$ is unique.

\end{proof}

From this theorem it follows that a given  simplicial cone $K\subset \R^n$
determines a partition of the space $\R^n$ in $2^n$ cones in the sense that
$$\R^n=\bigcup_{I\subset N}\cone \{\ee_i,\uu_j:\;i\in  I,\;j\in  I^c\}$$
and for two different sets $ I$ of indices the respective cones do not contain common
interior points. The cones in the above union are exactly the sums of orthogonal faces.

This theorem suggests the following algorithm 
for
finding the projection $\Pp_K\x$:

Step 1. For the subset $I\subset N$ 
we solve the following linear system in $\alpha^i$
\begin{equation}\label{10}
	\x^\top\ee_\ell =\sum_{i\in I}\alpha^i\ee^\top_i\ee_\ell,\;l\in I.
\end{equation}

Step 2. Then, we select from the family of all
subsets in $N$ the subfamily $\Delta$ of subsets $I$ for
which the system possesses non-negative solutions.

Step 3. For each $I\in \Delta$ we solve
the linear system in $\beta^j$
\begin{equation}\label{11}
	\x^\top\uu_k = \sum_{j\in I^c}\beta^j\uu_j^\top\uu_k,\;k\in I^c.
\end{equation}
By Theorem \ref{lattproj} among these systems
there exists exactly one with non-negative solutions.
By this theorem, for corresponding $I$ and for the solution
of the system (\ref{10}), we must have
$$\Pp_K\x=\sum_{i\in I}\alpha^i \ee_i.$$

This algorithm requires that we solve $2^n$ linear systems of at
most $n$ equations in Step 1 (\ref{10}) and another $2^{|\Delta|}$ systems in Step 2 
(\ref{11}).
(Observe that all these systems are given by Gram
matrices, hence they have unique solutions.)
Perhaps this great number of systems can be substantially reduced, but it still remains 
considerable. 

\begin{remark}\label{r1} If $K$ is subdual; that is, 
	if $\ee^\top_k\ee_\ell \geq 0,\; k,\;l\in N$, the above algorithm can be reduced as 
	follows: By supposing that we have got the representation (\ref{5}) of $x$ with
	non-negative coefficients, we multiply both sides of (\ref{5}) by an 
	arbitrary $\ee_l^\top$. If $\x^\top\ee_l <0$ then
	$l$ cannot be in $I$, otherwise the relations $\uu_j\ee_l=0,\;j\in I^c$
	and $\ee^\top_i \ee_l \geq 0,\;i\in I$ would furnish a contradiction.
	Thus, we have to look for the set $I$ of indices (for which we have to solve the 
	system (\ref{10})) among the subfamilies of
	$\{i\in N: \x^\top\ee_i \geq 0\}.$ (Arguments like this can be used,
    as it was done e. g. in \cite{Morillas2005} for the Dykstra algorithm, 
    to eliminate some hyperplanes
    while comkputing successive approximations of the solution.)

	Obviously, the proposed method is inefficient.
	It was presented by A. B. N\'emeth and S. Z. N\'emeth in 
	\cite{NemethNemeth2009} as a preparatory matherial for an efficient
	algorithm for so called \emph{isotone projection cones} only. For isotone 
	projection cones we can obtain the projection of a point in at most $n$ steps, 
	where $n$ is the dimension of the space. Isotone projection cones are special  
	simplicial cones. Even if there are important isotone projection cones in 
	applications, they are rather particular in the family of  simplicial cones.
	\end{remark}
%**********************************************************

\section{Heuristic method for projection onto a simplicial cone}\label{sec:alg}

Regardless the inconveniences of the above presented exact method, which follow 
from its combinatorial character, it suggests an interesting heuristic algorithm.
To explain its intuitive background we consider again the  simplicial cone
$$K=\cone\{\ee_1,\dots,\ee_n\}$$
and its polar
$$K^\circ =\cone\{\uu_1,\dots,\uu_n\}$$
given by Lemma \ref{polar}.

Take an arbitrary $\x\in \R^n$. We are seeking the projection $\Pp_Kx.$

If $\ee_i^\top\x\leq 0,\;\forall \;i\in N$, then $\x\in K^\circ =
\ker \Pp_K$, hence $\Pp_K\x=0.$

If $\uu_j^\top\x \leq 0,\;\forall \;j\in N$, then $\x\in K$,
and hence $\Pp_K\x=\x$.

We can assume that $\x\notin K\cup K^\circ$. Hence, $\x$
projects in a proper face of $K$ and in a proper face of $K^\circ.$

Take an arbitrary family $I\subset N$ of indices. Then, the vectors
$$\ee_i,\,\uu_j:\;i\in I,\,j\in I^c$$
entgender by Corollary \ref{eiujbaz} a reference system in $\R^n$.
Then, 
\begin{equation}\label{heur}
\x=\sum_{i\in I}\alpha^i\ee_i+\sum_{j\in I^c}\beta^j\uu_j
\end{equation}
with some $\alpha^i,\;\beta^j\in \R.$
(As far as the family $I\subset N$ of indices is given, we can determine
the coefficients $ \alpha^i$ and $\beta^j$, according to
Theorem \ref{lattproj}, by solving the systems (\ref{10}) and (\ref{11}).)

If we have $\alpha^i\geq 0,\;\beta^j\geq 0:\;i\in I,\;j\in I^c$,
then from Theorem \ref{lattproj} we obtain
$$\Pp_K\x=\sum_{i\in I}\alpha^i\ee_i$$
and
$$\Pp_{K^\circ}\x=\sum_{j\in I^c}\beta^j\uu_j.$$

In this case $\x$  is projected onto face $F=\cone \{\ee_i:\;i\in I\}$
ortogonally along the subspace engendered by the elements
$\{\uu_j:\;j\in I^c\}$, roughly speaking, along the
orthogonal face $F^\perp$ of $F$. 

Suppose that $\beta^j<0$ for some $j\in I^c$. Then, 
considering the reference system entgendered by $\ee_i,\;\uu_j,\;i\in I,\;j\in I^c$,
$\x$ lies in its orthant with negative $j^{th}$ coordinate, that is in the direction
of the vector $-\uu_j$. By construction, $\ee_j$ and $\uu_j$ form an
obtuse angle. Hence the angle of $\ee_j$ and $-\uu_j$ is an acute one.
Thus there is a real chance that in a new reference system in which
$\ee_j$ replaces $\uu_j$, the coordinate of $\x$ with respect to $\ee_j$
has the same sign as its coordinate with respect to $-\uu_j$,
that is positive (or at least non-negative).

If we have $\alpha^i<0$ for some $i\in I$, then by similar reasoning 
it seems to be advantageous to replace
$\ee_i$ with $\uu_i$, and so on.

Thus, we arrive to the following step in our algorithm:

Substitute $\uu_j$ with $\ee_j$ if $\beta^j<0$ and substitute
$\ee_i$ with $\uu_i$ if $\alpha^i<0$ and solve the systems
(\ref{10}) and (\ref{11}) for the new configuration of
indices $I$. We shall call this step an \textbf{iteration} of the heuristic algorithm.

Then, repeat the procedure for the new configuration of $I$ and so on,
until we obtain a representation (\ref{heur}) of $\x$ with all the coefficients
non-negative.

\begin{verbatim}

\end{verbatim}

\section{Experimental results}

The heuristic algorithm was programmed in Scilab, an open source platform for numerical computation.\footnote{http://www.scilab.org/} Experiments were performed on numerical examples for $2,\,3,\,5,\,10,\,15,\,20,\,25,\,30,\,50,\,75,\,100,\,200,\,300,\,500$ dimensional cones. The algorithm was performed on $100000$ random examples for each of the problem sizes $2,\dots,\,100$. Statistical analysis on a subset of $10000$ examples from the set of $100000$ examples for size $100$ indicates no significant difference in overall results and performance, therefore we subsequenlty reduced the number of experiments on larger problem sizes.     $10000$ random examples were used for sizes $200$ and $300$ and $1000$ examples for size $500$, as the time needed by the algorithm increases with size.
\begin{table}[h!]
\caption{Number of changes, iterations, iterations where the number of changes increased and loops for the various cone dimensions}
%\vspace*{-3mm}
\label{tab:1}
\begin{center}
\begin{tabular}{|c|c|c|c|c|}
\hline
\textbf{Size} & \textbf{Changes}       & \textbf{Iterations} & \textbf{Iterations with} & \textbf{ Loops}\\
&&& \textbf{increases [\%]}& \textbf{[\%]}\\
\hline
$2$ & $1$                  & $1$ & $0$ & $4.382$\\ 
$3$ & $2$                  & $1$ & $  3.9\pm 0.1$ & $4.278$\\
$5$ & $4$                  & $2$ & $  13\pm 0.2$ & $1.396$\\
$10$ & $11$                & $3$ & $  26\pm 0.3$ & $0.273$\\
$15$ & $17$                & $4$ & $  30.3\pm 0.3$ & $0.029$\\
$20$ & $24$                & $4$ & $  31.6\pm 0.3$ & $0.007$\\
$25$ & $30$                & $4$ & $  31.2\pm 0.3$ & $0.003$\\
$30$ & $37$                & $4$ & $  29.9\pm 0.3$ & $-$\\
$50$ & $64$                & $5$ & $  26.8\pm 0.3$ & $-$\\
$75$ & $97$                & $5$ & $  24.2\pm 0.3$ & $-$ \\
$100$ & $131$              & $5$ & $  23.8\pm 0.3$ & $-$\\
$200$ & $267 \pm 1$ & $6$ & $  19.9\pm 0.8$ & $-$\\
$300$ & $409 \pm 1$ & $6$ & $  26.2\pm 0.9$ & $-$\\
$500$ & $700 \pm 5$ & $7$ & $  25.7\pm 2.7$ & $-$\\
\hline
\end{tabular}
\end{center}
\end{table}
Table \ref{tab:1} shows the experimental results. For each problem size, the averages of all runs are shown, together with a confidence interval at confidence level $95\%$ where appropriate.\footnote{Any difference less than $\pm0.5$ for integers and $\pm0.1$ for percentages, respectively, is not shown, as deemed irrelevant for the analysis.} The \textbf{Changes} column indicates the total number of swaps $\uu_j$ for $\ee_j$ and $\ee_j$ for $\uu_j$, respectively, before reaching the solution. The \textbf{Iterations} column indicates the number of iterations (as defined in Section \ref{sec:alg}) the algorithm performed before reaching a solution. The \textbf{Iterations with increases} column shows the percentage of iterations where the number of changes increased from the previous iteration. We noticed that in the majority of iterations the number of changes decreased, which led to the quick convergence of the algorithm in the vast majority of cases. In all examples the starting point for the search was the $\ee_1,\dots,\,\ee_n$ base. The final column shows the percentage of problems where the algorithm was aborted due to going in a loop by allocating in some iteration a set of $\ee_j$s and $\uu_j$s that were encountered in a previous iteration. The percentage of loops was exponentially decreasing as the size increased and we did not observe any loops in any experiments on problem sizes of $30$ or above. Overall, loops were observed in $0.1\%$ of the experiments, so the heuristic algorithm was successful $99.9\%$ of the time. A solution that we see for solving the problems that lead to a loop is to restart the algorithm from a different initial set of $\ee_j$s and $\uu_j$s.  The problems ending in a loop were excluded from the detailed analysis that follows.

More detailed analysis of the three main performance indicators of changes, iterations and iterations with increases was performed using boxplots as shown in Figures \ref{fig:1}, \ref{fig:2} and \ref{fig:3}. Although the total number of changes performed increases linearly with problem size (at a rate of less than $2 \times n$, even if considering maximum numer of changes), this does not affect the performance substantially (see Figure \ref{fig:1}, where the results were split into two parts for a clearer view).
\begin{figure}[h!]
\vspace*{-0.5cm}
\hspace*{-1.2cm}
\includegraphics[width=7.5cm]{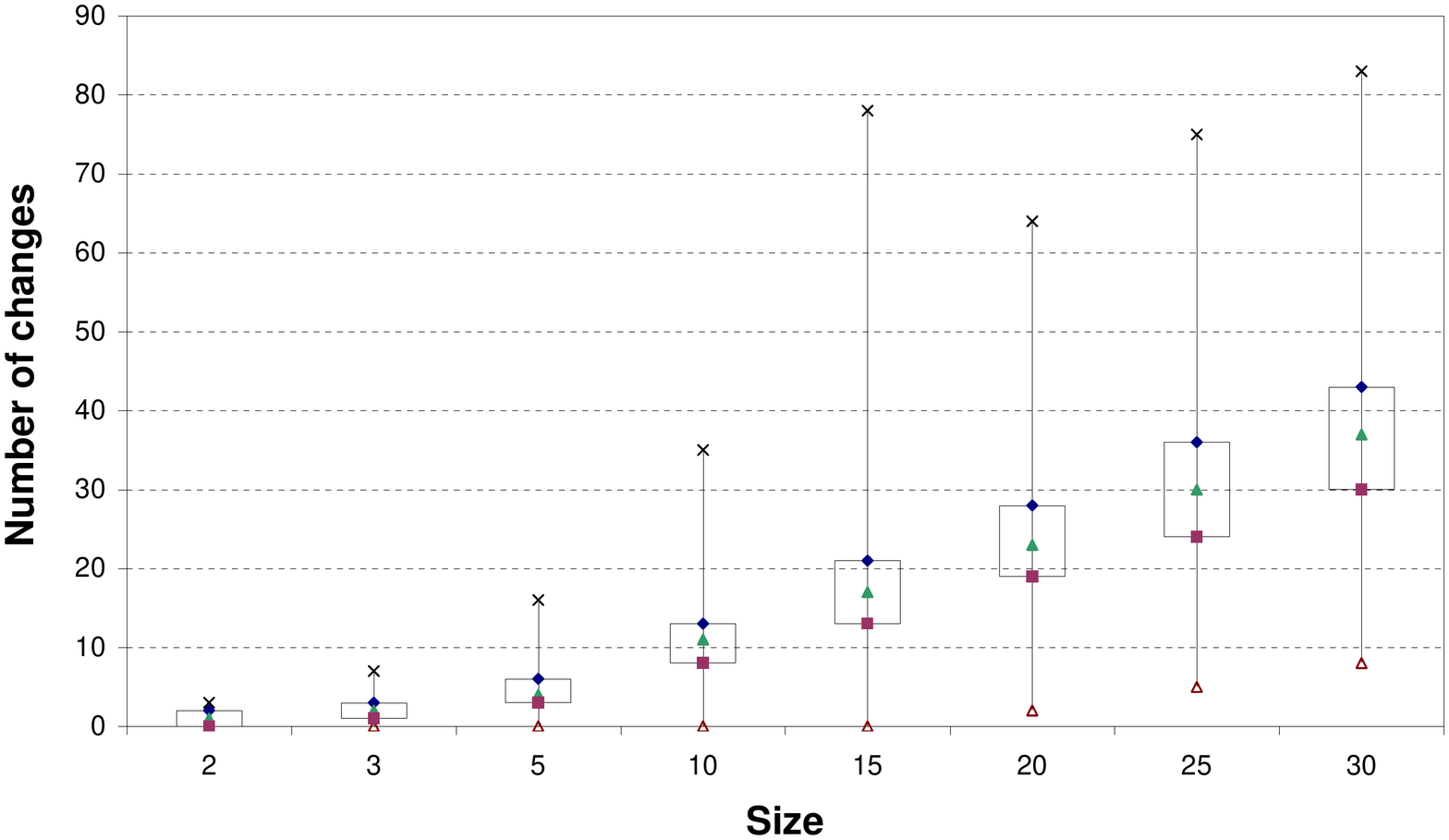}
\hspace*{-0.8cm}
\includegraphics[width=7.5cm]{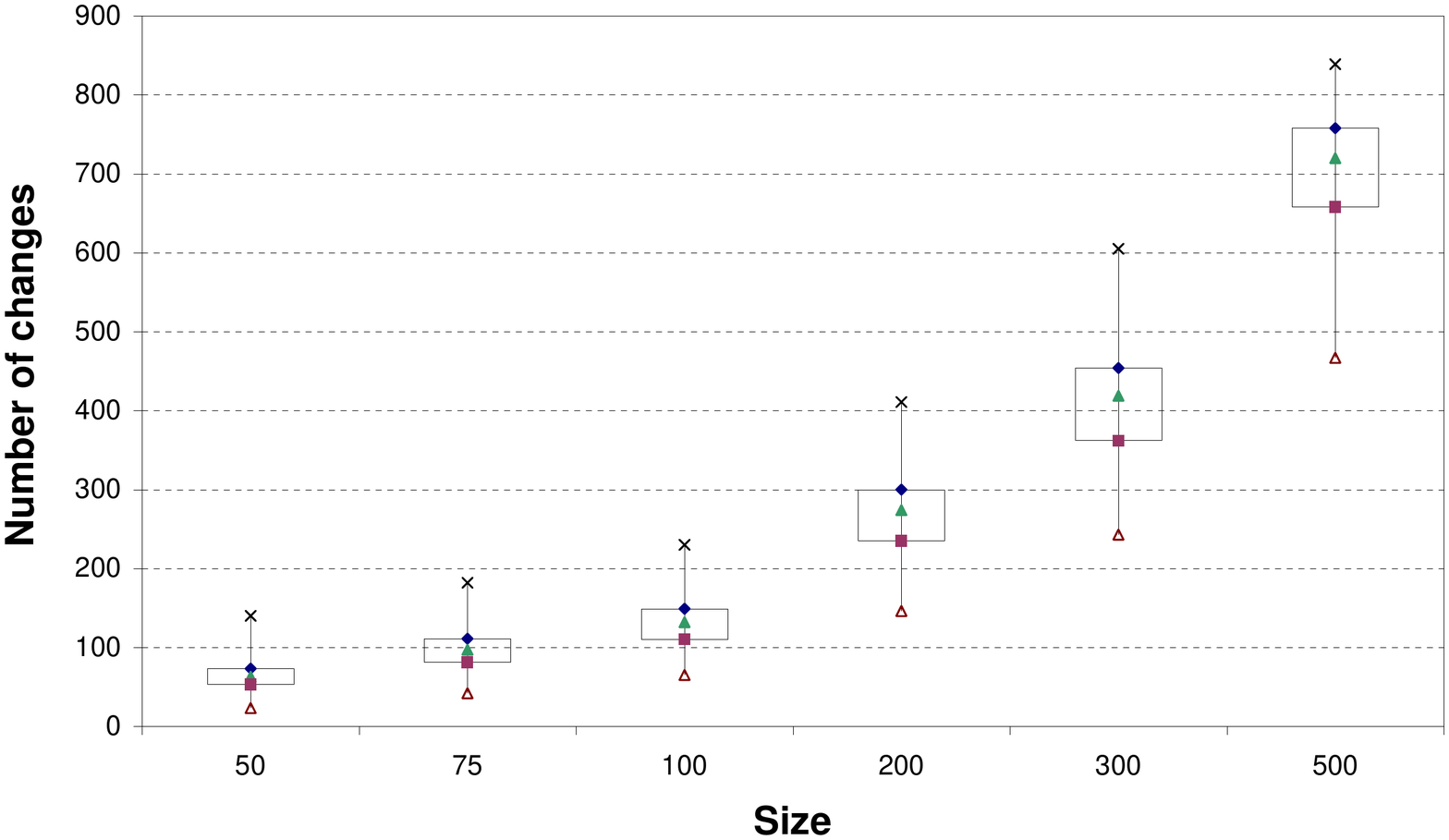}
\vspace*{-1.4cm}
\caption{Boxplots of number of changes performed for the various cone dimensions}
\label{fig:1}
\end{figure}
%\begin{figure}[h!]
%\includegraphics[width=10cm]{changes.pdf}
%\caption{Boxplots of number of changes performed for the various cone dimensions}
%\end{figure}

The number of iterations is the crucial indicator of performance. As shown in Figure \ref{fig:2} the number of iterations reaches at most $11$ (for sizes $15$ and $20$), but in $75\%$ of cases has a value of $7$ or below. We ran a few experiments on larger sizes, up to $1750$, and the largest number of iterations we obeserved was $13$. Running experiments on very large problem sizes is problematic due to computer memory limitations and the Scilab built-in solving of linear systems.\footnote{ Note that the time needed by one iteration substantially increases with problem size $n$ as one iteration involves solving a linear system with $n$ equations and $n$ variables.} \textit{The major benefit of this heuristic algorithm is the small number of iterations even for very large number of cone dimensions.}
\begin{figure}[h!]
\begin{center}
\includegraphics[width=12cm]{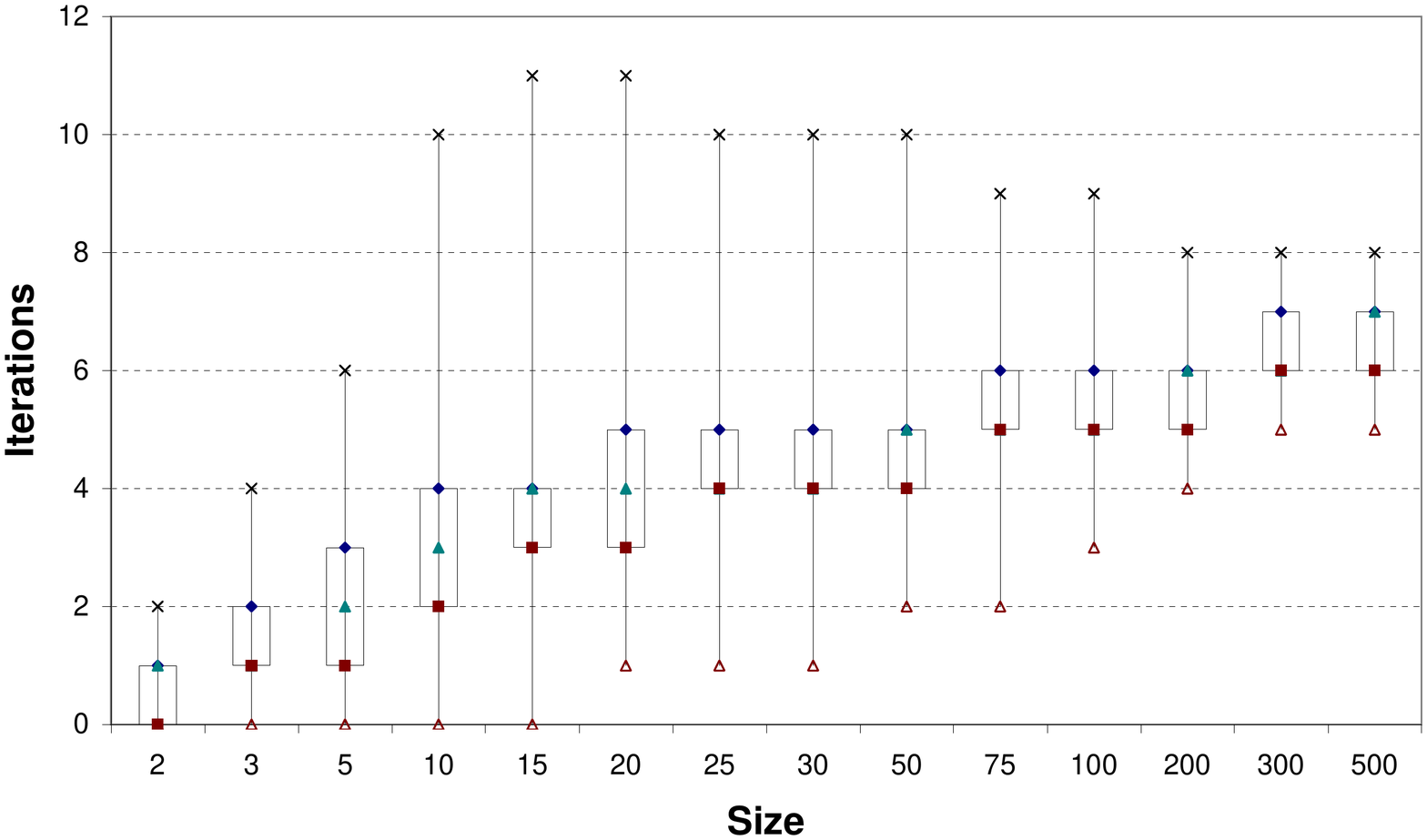}
\vspace*{-1.4cm}
\end{center}
\caption{Boxplots of number of iterations needed for the various cone dimensions}
\label{fig:2}
\end{figure}

As our heuristic algorithm seems to converge quickly, we wanted to know how frequently it deviates from the optimal path. An optimal path would consist of iterations with decreasing number of changes. Figure \ref{fig:3} shows the boxplots for the number of iterations where an increase in the number of changes took place. The maximum number of such iterations over all experiments is $4$, but $75\%$ of examples involved only one or no such iteration. This provides an explanation for the fast convergence: the algorithm very rarely deviates from the optimal path.
\begin{figure}[h!]
\begin{center}
\includegraphics[width=12cm]{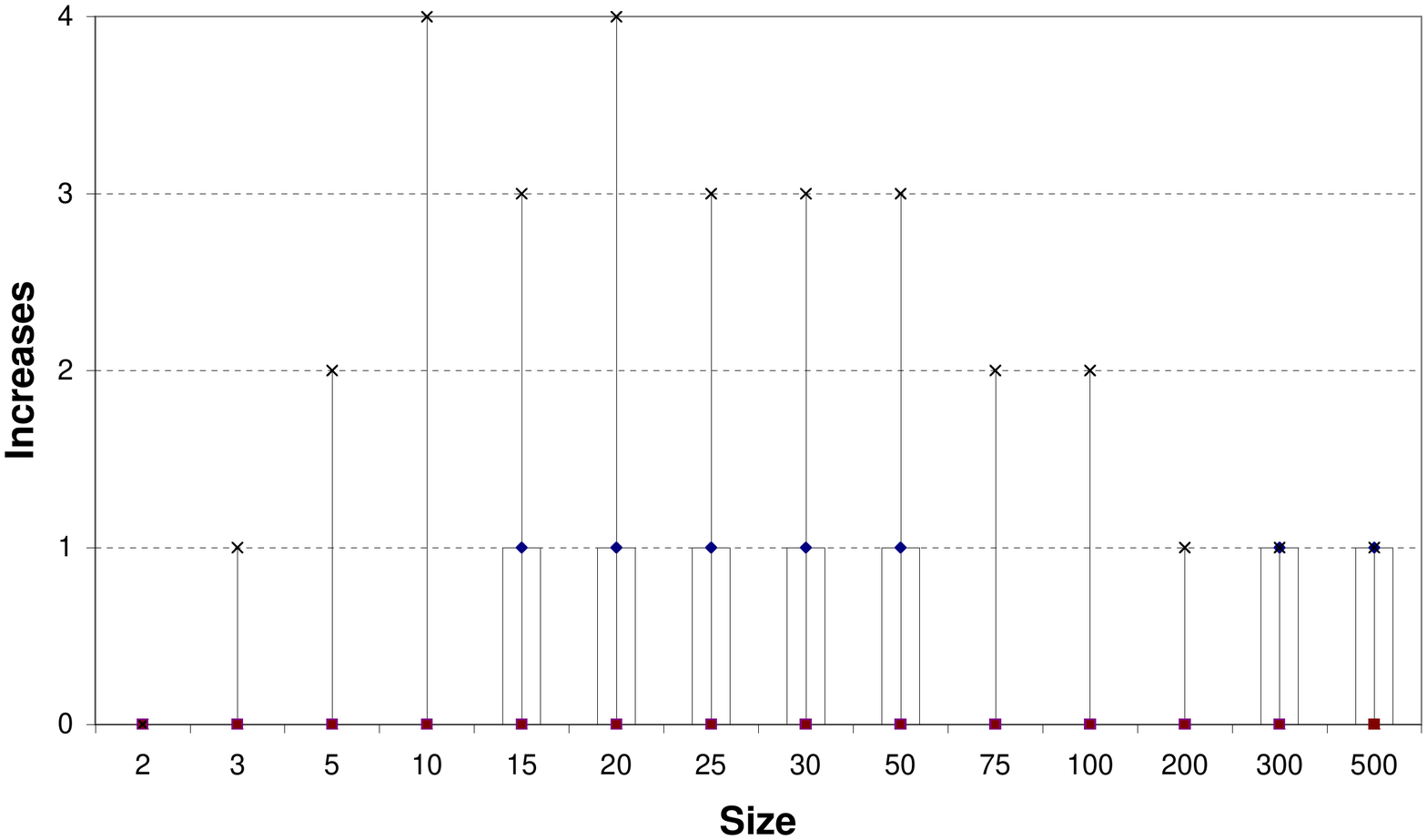}
\vspace*{-1.4cm}
\end{center}
\caption{Boxplots of number of iterations with increases in number of changes needed for the various cone dimensions}
\label{fig:3}
\end{figure}

\section{Conclusion}

We presented a heuristic method of projection on simplicial cones based on Moreau's 
decomposition theorem.
The heuristic algorithm presented in this note iteratively finds the projection onto a simplicial cone in a surprisingly small number of steps even for large cone dimensions in $99.9.\%$ of the cases. We attribute the success to the fact that the algorithm rarely deviates from the optimal path, in every iteration it usually has to change less base values than in the previous iteration.
We are planning to further extend the algorithm with random restart hoping to achieve $100\%$ success rate.  

\section*{Acknowledgements}
	
S. Z. N\'emeth was supported by the Hungarian Research Grant OTKA 60480. The authors are grateful to J. 
Dattorro for many helpful conversations.
%\clearpage

\bibliographystyle{plain}
\bibliography{projlattco}

\begin{thebibliography}{10}

\bibitem{Dattorro2005}
J.~Dattorro.
\newblock {\em Convex Optimization $\&$ Euclidean Distance Geometry}.
\newblock {$\mathcal{M}\varepsilon\beta oo$}, 2005, v2009.04.11.

\bibitem{DeutschHundal1994}
F.~Deutsch and H.~Hundal.
\newblock The rate of convergence of {D}ykstra's cyclic projections algorithm:
  the polyhedral case.
\newblock {\em Numer. Funct. Anal. Optim.}, 15(5-6):537--565, 1994.

\bibitem{Dykstra1983}
R.~L. Dykstra.
\newblock An algorithm for restricted least squares regression.
\newblock {\em J. Amer. Stat. Assoc.}, 78(384):273--242, 1983.

\bibitem{IsacNemeth1990c}
G.~Isac and A.~B. N\'emeth.
\newblock Projection methods, isotone projection cones, and the complementarity
  problem.
\newblock {\em J. Math. Anal. Appl.}, 153(1):258--275, 1990.

\bibitem{MingGuo-LiangHong-BinKaiWang2007}
T.~Ming, T.~Guo-Liang, F.~Hong-Bin, and Ng.~Kai Wang.
\newblock A fast {EM} algorithm for quadratic optimization subject to convex
  constraints.
\newblock {\em Statist. Sinica}, 17(3):945--964, 2007.

\bibitem{Moreau1962}
J.~J. Moreau.
\newblock D\'ecomposition orthogonale d'un espace hilbertien selon deux c\^ones
  mutuellement polaires.
\newblock {\em C. R. Acad. Sci.}, 255:238--240, 1962.

\bibitem{Morillas2005}
P.~M. Morillas.
\newblock Dykstra's algorithm with strategies for projecting onto certain
  polyhedral cones.
\newblock {\em Applied Mathematics and Computation}, 167(1):635--649, 2005.

\bibitem{NemethNemeth2009}
A.~B. N\'emeth and S.~Z. N\'emeth.
\newblock How to project onto an isotone projection cone.
\newblock {\em Linear Algebra Appl.}, submitted.

\bibitem{Nemeth2008}
S.~Z. N\'emeth.
\newblock Iterative methods for nonlinear complementarity problems on isotone
  projection cones.
\newblock {\em J. Math. Anal. Appl.}, 350(1):340--347, 2009.

\bibitem{Shusheng2000}
X.~Shusheng.
\newblock Estimation of the convergence rate of {D}ykstra's cyclic projections
  algorithm in polyhedral case.
\newblock {\em Acta Math. Appl. Sinica (English Ser.)}, 16(2):217--220, 2000.

\end{thebibliography}
\end{document}